\documentclass[preprint,12pt]{elsarticle}

\usepackage{amssymb}
\usepackage{amsthm}

\newcommand{\sysn}{\left\{\begin{array}{rcl}}
\newcommand{\sysk}{\end{array}\right.}

\newtheorem{theorem}{Theorem}[section]
\newtheorem{lemma}[theorem]{Lemma}

\theoremstyle{example}

\newtheorem{proposition}[theorem]{Proposition}
\theoremstyle{definition}
\newtheorem{definition}[theorem]{Definition}

\newtheorem{corollary}[theorem]{Corollary}

\journal{...}

\begin{document}

\begin{frontmatter}

\title{Selectors for sequences of subsets of hyperspaces}

\author{Alexander V. Osipov}

\ead{OAB@list.ru}


\address{Krasovskii Institute of Mathematics and Mechanics, Ural Federal
 University,

 Ural State University of Economics, 620219, Yekaterinburg, Russia}

\begin{abstract}
For a Hausdorff space $X$ we denote by $2^X$ the family of all
closed subsets of $X$. In this paper we continue to research
relationships between closure-type properties of hyperspaces over
a space $X$ and covering properties of $X$. We investigate
selectors for sequence of subsets of the space $2^X$  with the
${\bf Z^+}$-topology and the upper Fell topology (${\bf
F^+}$-topology). Also we consider the selection properties of the
bitopological space $(2^X,{\bf F^+}, {\bf Z^+}).$

\end{abstract}


\end{frontmatter}

\section{Introduction}
\label{}

Given a Hausdorff space $X$ we denote by $2^X$ the family of all
closed subsets of $X$. If $A$ is a subset of $X$ and $\mathcal{A}$
a family of subsets of $X$, then

$A^c=X\setminus A$ and $\mathcal{A}^c=\{A^c : A\in \mathcal{A}\}$,

$A^{-}=\{F\in 2^{X} : F\cap A\neq \emptyset \}$,

$A^+=\{F\in 2^X : F\subset A \}$.

Let $\Delta$ be a subset of $2^X$ closed for finite unions and
containing all singletons. We consider the next important cases:

$\bullet$ $\Delta$ is the collection $CL(X)=2^X\setminus
\{\emptyset\}$;

$\bullet$ $\Delta$ is the family $\mathbb{K}(X)$ of all non-empty
compact subsets of $X$;

$\bullet$ $\Delta$ is the family $\mathbb{F}(X)$ of all non-empty
finite subsets of $X$.

For $\Delta\subset 2^X$, the {\it upper $\Delta$-topology},
denoted by $\Delta^+$, is the topology whose base is the
collection $\{(D^c)^+ : D\in \Delta \}\cup \{2^X\}$.

When $\Delta=CL(X)$ we have the well-known {\it upper Vietoris
topology} ${\bf V^+}$, when $\Delta=\mathbb{K}(X)$ we have the
{\it upper Fell topology} (known also as the co-compact topology)
${\bf F^+}$, and when $\Delta=\mathbb{F}(X)$ we have the ${\bf
Z^+}$-topology.

Many topological properties are defined or characterized in terms
 of the following classical selection principles (\cite{koc2,sash,ts1}).
 Let $\mathcal{A}$ and $\mathcal{B}$ be sets consisting of
families of subsets of an infinite set $X$. Then:

$S_{1}(\mathcal{A},\mathcal{B})$ is the selection hypothesis: for
each sequence $(A_{n}: n\in \mathbb{N})$ of elements of
$\mathcal{A}$ there is a sequence $( b_{n} : n\in \mathbb{N})$
such that for each $n$, $b_{n}\in A_{n}$, and $\{b_{n}:
n\in\mathbb{N} \}$ is an element of $\mathcal{B}$.

$S_{fin}(\mathcal{A},\mathcal{B})$ is the selection hypothesis:
for each sequence $( A_{n}: n\in \mathbb{N})$ of elements of
$\mathcal{A}$ there is a sequence $( B_{n}: n\in \mathbb{N})$ of
finite sets such that for each $n$, $B_{n}\subseteq A_{n}$, and
$\bigcup_{n\in\mathbb{N}}B_{n}\in\mathcal{B}$.

In this paper, by a cover we mean a nontrivial one, that is,
$\mathcal{U}$ is a cover of $X$ if $X=\bigcup \mathcal{U}$ and
$X\notin \mathcal{U}$.

An open cover of a space is {\it large} if each element of the
space belongs to infinitely many elements of the cover.

An open cover $\mathcal{U}$ of a space $X$ is called:

$\bullet$  an $\omega$-cover (a $k$-cover) if each finite
(compact) subset $C$ of $X$ is contained in an element of
$\mathcal{U}$;

$\bullet$  a $\gamma$-cover (a $\gamma_k$-cover) if $\mathcal{U}$
is infinite and for each finite (compact) subset $C$ of $X$ the
set $\{U\in \mathcal{U} : C\nsubseteq U\}$ is finite.

Because of these definitions all spaces are assumed to be {\it
Hausdorff non-compact}, unless otherwise stated.

Let us mention that any $\omega$-cover ($k$-cover) is infinite and
large, and that any infinite subfamily of a $\gamma$-cover
($\gamma_k$-cover) is also a $\gamma$-cover ($\gamma_k$-cover).

For a topological space $X$ we denote:

$\bullet$ $\mathcal{O}$ --- the family of all open covers of $X$;

$\bullet$ $\Gamma$ --- the family of all open $\gamma$-covers of
$X$;

$\bullet$ $\Gamma_k$ --- the family of all open $\gamma_k$-covers
of $X$;

$\bullet$ $\Omega$ --- the family of all open $\omega$-covers of
$X$;

$\bullet$ $\mathcal{K}$ --- the family of all open $k$-covers of
$X$.

\bigskip

 Different $\Delta$-covers ($k$-covers, $\omega$-covers, $k_F$-covers, $c_F$-covers,...) exposed many dualities in hyperspace topologies such as co-compact topology ${\bf F^+}$, co-finite topology ${\bf
 Z^+}$, Pixley-Roy topology, Fell topology and Vietoris topology. They also play important roles in
 selection principles ([1,4-9,12-16]).

In this paper we continue to research relationships between
closure-type properties of hyperspaces over a space $X$ and
covering properties of $X$. We investigate selectors for sequence
of subsets of the space $2^X$  with the ${\bf Z^+}$-topology and
the upper Fell topology (${\bf F^+}$-topology). Also we consider
the selection properties of the bitopological space $(2^X,{\bf
F^+}, {\bf Z^+}).$

\section{Main definitions and notation}

The following lammas will be often used throughout the paper,
sometimes without explicit reference.

\begin{lemma}(Lemma 1 in \cite{cmkm})\label{l1} Let $Y$ be an open subset of a
space $X$ and $\mathcal{U}$ an open cover of $Y$. Then the
following holds:

(1)  $\mathcal{U}$ is a $k$-cover of $Y$ $\Leftrightarrow$ $Y^c\in
Cl_{F^+}(\mathcal{U}^c)$;

(2) $\mathcal{U}$ is an $\omega$-cover of $Y$ $\Leftrightarrow$
$Y^c\in Cl_{Z^+}(\mathcal{U}^c)$.

\end{lemma}

\begin{lemma}(Lemma 2 in \cite{cmkm})\label{lem2} Let $X$ be a topological
space, $Y$ an open subsets of $X$ and $\mathcal{U}=\{U_n: n\in
\mathbb{N}\}$ an open cover of $Y$. Then

(1) $\mathcal{U}$ is a $\gamma_k$-cover of $Y$ $\Leftrightarrow$
the sequence $(U^c_n : n\in \mathbb{N})$ converges to $Y^c$ in
$(2^X, {\bf F}^+)$;

(2) $\mathcal{U}$ is an $\gamma$-cover of $Y$ $\Leftrightarrow$
the sequence $(U^c_n : n\in \mathbb{N})$ converges to $Y^c$ in
$(2^X, {\bf Z}^+)$.
\end{lemma}

\begin{lemma}(Lemma 3 in \cite{cmkm})\label{l3} Given a space $X$ and an
open cover $\mathcal{U}$ of $X$ the following holds:

(1) $\mathcal{U}$ is a $k$-cover of $X$ $\Leftrightarrow$
$\mathcal{U}^c$ is a dense subset of $(2^X, {\bf F}^+)$;

(2) $\mathcal{U}$ is an $\omega$-cover of $X$ $\Leftrightarrow$
$\mathcal{U}^c$ is a dense subset of $(2^X, {\bf Z}^+)$.

\end{lemma}

Let $X$ be a topological space, and $x\in X$. A subset $A$ of $X$
{\it converges} to $x$, $x=\lim A$, if $A$ is infinite, $x\notin
A$, and for each neighborhood $U$ of $x$, $A\setminus U$ is
finite. Consider the following collection:

$\bullet$ $\Omega_x=\{A\subseteq X : x\in \overline{A}\setminus
A\}$;

$\bullet$ $\Gamma_x=\{A\subseteq X : x=\lim A\}$.

Note that if $A\in \Gamma_x$, then there exists $\{a_n\}\subset A$
converging to $x$. So, simply $\Gamma_x$ may be the set of
non-trivial convergent sequences to $x$.

\begin{definition}  Let $X$ be a space and let $\mathcal{U}=\{U_{\alpha}: \alpha\in \Lambda\}$ be an open
cover of $X$. Then {\it $\mathcal{U}^c=\{U^c_{\alpha}: \alpha\in
\Lambda\}$ converges to $\{\emptyset\}$} in $(2^X, \tau)$ where
$\tau$ is a topology on $2^X$, if for every $F\in 2^X$ the
$\mathcal{U}^c$ converges to $F$, i.e. for each neighborhood $W$
of $F$ in the space $(2^X, \tau)$, $|\{\alpha :
U_{\alpha}^c\nsubseteq W, \alpha\in \Lambda \}|<\aleph_0$.
\end{definition}

\begin{lemma}\label{lem5} Let $X$ be a space and let $\mathcal{U}=\{U_{\alpha}: \alpha\in \Lambda\}$ be an open
cover of $X$. Then the following are equivalent:
\begin{enumerate}

\item $\mathcal{U}$ is an $\gamma$-cover of $X$;

\item $\mathcal{U}^c$ converges to $\{\emptyset\}$ in $(2^X, {\bf
Z}^+)$.

\end{enumerate}

\end{lemma}

For a topological space $(2^X, \tau)$ we denote:

$\bullet$ $\mathcal{D}_{\Omega}$ --- the family of dense subsets
of $(2^X, \tau)$;

$\bullet$ $\mathcal{D}_{\Gamma}$ --- the family of converging to
$\{\emptyset\}$  subsets of $(2^X, \tau)$.

Since every $\gamma$-cover contains a countably $\gamma$-cover,
then each converging to $\{\emptyset\}$ subset of $(2^X,{\bf
Z^+})$ contains a countable converging to $\{\emptyset\}$ subset
of $(2^X,{\bf Z^+})$.

\section{Hyperspace $(2^X,{\bf Z^+})$}

\begin{theorem}\label{th11} Assume that $\Phi, \Psi\in \{\Gamma, \Omega\}$, $\star\in \{1,fin\}$.  Then for a space $X$ the following statements are equivalent:

\begin{enumerate}

\item $X$ satisfies $S_{\star}(\Phi, \Psi)$;

\item $(2^X,{\bf Z^+})$ satisfies $S_{\star}(\mathcal{D}_{\Phi},
\mathcal{D}_{\Psi})$.

\end{enumerate}

\end{theorem}

\begin{proof} We prove the theorem for $\star=fin$, the other
proofs being similar.

$(1)\Rightarrow(2)$. Let $(D_i : i\in \mathbb{N})$ be a sequence
of dense subsets of $(2^X,{\bf Z^+})$ such that $D_i\in
\mathcal{D}_{\Phi}$ for each $i\in \mathbb{N}$. Then $(D^c_i: i\in
\mathbb{N})$ is a sequence of open covers of $X$ such that
$D^c_i\in \Phi$ for each $i\in \mathbb{N}$. Since $X$ satisfies
$S_{fin}(\Phi, \Psi)$, there is a sequence $(A_{i}: i\in
\mathbb{N})$ of finite sets such that for each $i$,
$A_{i}\subseteq D^c_i$, and $\bigcup_{i\in\mathbb{N}}A_{i}\in
\Psi$. It follows that $\bigcup_{i\in\mathbb{N}}A^c_{i}\in
\mathcal{D}_{\Psi}$.

$(2)\Rightarrow(1)$. Let $(\mathcal{U}_n : n\in \mathbb{N})$ be a
sequence of open covers of $X$ such that $\mathcal{U}_n\in \Phi$.
For each $n$, $\mathcal{A}_n:=\mathcal{U}^c_n$ is a dense subset
of $(2^X,{\bf Z^+})$ such that $\mathcal{A}_n\in
\mathcal{D}_{\Phi}$. Applying that $(2^X,{\bf Z^+})$ satisfies
$S_{fin}(\mathcal{D}_{\Phi}, \mathcal{D}_{\Psi})$, there is a
sequence $(A_{n}: n\in \mathbb{N})$ of finite sets such that for
each $n$, $A_{n}\subseteq \mathcal{A}_n$, and
$\bigcup_{n\in\mathbb{N}}A_{n}\in \mathcal{D}_{\Psi}$. Then
$\bigcup_{n\in\mathbb{N}} U_n $ is an open cover of $X$ where
$U_n=A^c_n$ for each $n\in \mathbb{N}$ and
$\bigcup_{n\in\mathbb{N}} U_n\in \Psi$.

\end{proof}

\begin{corollary}(Theorem 5 in \cite{mkm1}) For a space $X$ the following are equivalent:

\begin{enumerate}

\item $X$ satisfies $S_{1}(\Omega, \Omega)$;

\item $(2^X,{\bf Z^+})$ satisfies $S_{1}(\mathcal{D}_{\Omega},
\mathcal{D}_{\Omega})$.

\end{enumerate}

\end{corollary}

\begin{corollary}(Theorem 13 in \cite{mkm1}) For a space $X$ the following are equivalent:

\begin{enumerate}

\item $X$ satisfies $S_{fin}(\Omega, \Omega)$;

\item $(2^X,{\bf Z^+})$ satisfies $S_{fin}(\mathcal{D}_{\Omega},
\mathcal{D}_{\Omega})$.

\end{enumerate}

\end{corollary}

\begin{theorem}\label{th12} Assume that $\Phi, \Psi\in \{\Gamma, \Omega, \}$, $\star\in \{1,fin\}$. Then for a space $X$ the following statements are equivalent:

\begin{enumerate}

\item Each open set $Y\subset X$ has the property $S_{\star}(\Phi,
\Psi)$;

\item  For each $E\in 2^X$, $(2^X,{\bf Z^+})$ satisfies
$S_{\star}(\Phi_{E}, \Psi_{E})$.

\end{enumerate}

\end{theorem}

\begin{proof} We prove the theorem for $\star=1$, the other
proofs being similar.

$(1)\Rightarrow(2)$. Let $E\in 2^X$ and let $(\mathcal{A}_n : n\in
\mathbb{N})$ be a sequence
 such that $\mathcal{A}_n\in \Phi_{E}$ for each $n\in \mathbb{N}$.
  Then $(\mathcal{A}^c_n:
n\in \mathbb{N})$ is a sequence of open covers of $E^c$ such that
$\mathcal{A}^c_n\in \Phi$ for each $n\in \mathbb{N}$. Since $E^c$
has the property $S_{\star}(\Phi, \Psi)$, there is a sequence
$(A^c_n : n\in \mathbb{N})$ such that $A^c_n\in \mathcal{A}^c_n$
for each $n\in \mathbb{N}$ and $\{A^c_n : n\in \mathbb{N}\}$ is
open cover of $E^c$ such that $\{A^c_n : n\in
\mathbb{N}\}\in\Psi$. It follows that $\{A_n : n\in
\mathbb{N}\}\in  \Psi_{E}$.

$(2)\Rightarrow(1)$. Let $Y$ be an open subset of $X$ and let
$(\mathcal{F}_n : n\in \mathbb{N})$ be a sequence of open covers
of $Y$ such that $\mathcal{F}_n\in \Phi_Y$ (where $\Phi_Y$ is the
$\Phi$ family of covers of $Y$). Let $E=X\setminus Y$. Put
$\mathcal{A}_n=\mathcal{F}_n^c$
 for each $n\in \mathbb{N}$. Then $\mathcal{A}_n\subset 2^X$ and
$\mathcal{A}_n\in \Phi_{E}$ for each $n\in \mathbb{N}$. Since, by
(2), $(2^X,{\bf Z^+})$ satisfies $S_{1}(\Phi_{E}, \Psi_{E})$,
there is a sequence $(A_n : n\in \mathbb{N})$ such that $A_n\in
\mathcal{A}_n$ for each $n\in \mathbb{N}$ and $\{A_n : n\in
\mathbb{N}\}\in \Psi_{E}$. It follows that $\{F_n: F_n=A^c_n, n\in
\mathbb{N}\}\in \Psi$.

\end{proof}

 \begin{corollary}(Theorem 3 in \cite{koc1}) For a space $X$ the following
statements are equivalent:

\begin{enumerate}

\item Each open set $Y\subset X$ has the property $S_1(\Omega,
\Gamma)$;

\item $(2^X,{\bf Z^+})$  is Fr$\acute{e}$chet-Urysohn;

\item $(2^X,{\bf Z^+})$  is strongly Fr$\acute{e}$chet-Urysohn.

\end{enumerate}

\end{corollary}

\begin{corollary} (Theorem 1 in \cite{mkm1}) For a space $X$ the following
are equivalent:

\begin{enumerate}

\item Each open set $Y\subset X$ satisfies $S_1(\Omega,\Omega)$;

\item  $(2^X,{\bf Z^+})$ has countable strong fan tightness (For
each $E\in 2^X$, $(2^X,{\bf Z^+})$ satisfies
$S_{1}(\Omega_{E},\Omega_{E}))$.

\end{enumerate}

\end{corollary}

\begin{corollary} (Theorem 9 in \cite{mkm1}) For a space $X$ the following are
equivalent:

\begin{enumerate}

\item Each open set $Y\subset X$ satisfies
$S_{fin}(\Omega,\Omega)$;

\item  $(2^X,{\bf Z^+})$ has countable fan tightness (For each
$E\in 2^X$, $(2^X,{\bf Z^+})$ satisfies
$S_{fin}(\Omega_{E},\Omega_{E}))$.

\end{enumerate}

\end{corollary}

Recall that a space is {\it perfect} if every open subset is an
$F_{\sigma}$-subset \cite{hm}. Clearly every semi-stratifiable
space is perfect.

Note that all properties in the Scheepers Diagram
(\cite{jmss,sch3}) are hereditary for $F_{\sigma}$-subsets
(Corollary 2.4 in \cite{ot}).

\begin{proposition} Assume that $\Phi, \Psi\in \{\Gamma, \Omega, \}$, $\star\in \{1,fin\}$. Then for a perfect topological
space $X$ the following statements are equivalent:

\begin{enumerate}

\item $X$ satisfies $S_{\star}(\Phi, \Psi)$;

\item Each open set $Y\subset X$ has the property $S_{\star}(\Phi,
\Psi)$.

\end{enumerate}

\end{proposition}

\begin{theorem} Assume that $\Phi, \Psi\in \{\Gamma, \Omega\}$, $\star\in \{1,fin\}$.  Then for a perfect space $X$ the following statements are equivalent:

\begin{enumerate}

\item $X$ satisfies $S_{\star}(\Phi, \Psi)$;

\item $(2^X,{\bf Z^+})$ satisfies $S_{\star}(\mathcal{D}_{\Phi},
\mathcal{D}_{\Psi})$;

\item  For each $E\in 2^X$, $(2^X,{\bf Z^+})$ satisfies
$S_{\star}(\Phi_{E}, \Psi_{E})$.

\end{enumerate}

\end{theorem}

Clearly that every perfectly normal space is perfect.

\begin{corollary} Assume that $\Phi, \Psi\in \{\Gamma, \Omega\}$, $\star\in \{1,fin\}$.  Then for a perfectly normal
space $X$ the following statements are equivalent:

\begin{enumerate}

\item $X$ satisfies $S_{\star}(\Phi, \Psi)$;

\item $(2^X,{\bf Z^+})$ satisfies $S_{\star}(\mathcal{D}_{\Phi},
\mathcal{D}_{\Psi})$;

\item  For each $E\in 2^X$, $(2^X,{\bf Z^+})$ satisfies
$S_{\star}(\Phi_{E}, \Psi_{E})$.

\end{enumerate}

\end{corollary}

\section{Hyperspace $(2^X,{\bf F^+})$}

Note that $(2^X,{\bf F^+})$ satisfies the selection principle
$S_{fin}(\mathcal{O}, \mathcal{O})$ for any space $X$, because
$(2^X,{\bf F^+})$ is always compact (see \cite{fe}).

$S_{fin}(\mathcal{O}, \mathcal{O})$ property is called the {\it
Menger property} (see \cite{jmss,sch3}).

\begin{lemma}\label{lem6} Let $X$ be a space and let $\mathcal{U}=\{U_{\alpha}: \alpha\in \Lambda\}$ be an open
cover of $X$. Then the following are equivalent:

\begin{enumerate}

\item $\mathcal{U}$ is an $\gamma_k$-cover of $X$;

\item $\mathcal{U}^c$ converges to $\{\emptyset\}$ in $(2^X, {\bf
F}^+)$.

\end{enumerate}

\end{lemma}

\begin{theorem} Assume that $\Phi, \Psi\in \{\Gamma_k, \mathcal{K}\}$, $\star\in \{1,fin\}$.
Then for a space $X$ the following statements are equivalent:

\begin{enumerate}

\item $X$ satisfies $S_{\star}(\Phi, \Psi)$;

\item $(2^X,{\bf F^+})$ satisfies $S_{\star}(\mathcal{D}_{\Phi},
\mathcal{D}_{\Psi})$.

\end{enumerate}

\end{theorem}

\begin{proof} The proof is similar to the proof of Theorem
\ref{th11}.

\end{proof}

\begin{corollary}(Theorem 4 in \cite{mkm1}) For a space $X$ the following are equivalent:

\begin{enumerate}

\item $X$ satisfies $S_{1}(\mathcal{K}, \mathcal{K})$;

\item $(2^X,{\bf F^+})$ satisfies $S_{1}(\mathcal{D}_{\Omega},
\mathcal{D}_{\Omega})$.

\end{enumerate}

\end{corollary}

\begin{corollary}(Theorem 12 in \cite{mkm1}) For a space $X$ the following are equivalent:

\begin{enumerate}

\item $X$ satisfies $S_{fin}(\mathcal{K}, \mathcal{K})$;

\item $(2^X,{\bf F^+})$ satisfies $S_{fin}(\mathcal{D}_{\Omega},
\mathcal{D}_{\Omega})$.

\end{enumerate}

\end{corollary}

\begin{theorem} Assume that $\Phi, \Psi\in \{\Gamma_k, \mathcal{K}\}$, $\star\in \{1,fin\}$. Then for a space $X$ the following statements are equivalent:

\begin{enumerate}

\item Each open set $Y\subset X$ has the property $S_{\star}(\Phi,
\Psi)$;

\item  For each $E\in 2^X$, $(2^X,{\bf F^+})$ satisfies
$S_{\star}(\Phi_{E}, \Psi_{E})$.

\end{enumerate}

\end{theorem}

\begin{proof}  The proof is similar to the proof of Theorem
\ref{th12}.

\end{proof}

\begin{corollary}(Theorem 23 in \cite{cmkm}) For a space $X$ the
following are equivalent:

\begin{enumerate}

\item  For each $E\in 2^X$, $(2^X,{\bf F^+})$ satisfies
$S_{1}(\Gamma_{E}, \Omega_{E})$;

\item Each open set $Y\subset X$ has the property $S_{1}(\Gamma_k,
\mathcal{K})$.

\end{enumerate}

\end{corollary}

\begin{corollary}(Theorem 32 in \cite{cmkm}) For a space $X$ the
following are equivalent:

\begin{enumerate}

\item  For each $E\in 2^X$, $(2^X,{\bf F^+})$ satisfies
$S_{1}(\Gamma_{E}, \Gamma_{E})$;

\item Each open set $Y\subset X$ has the property $S_{1}(\Gamma_k,
\Gamma_k)$.

\end{enumerate}

\end{corollary}

\begin{definition} A subset $A$ of a space $X$ is called an
{\it $k$-$F_{\sigma}$-set} if $A$ can be represented as
$A=\bigcup\limits_{i=1}^{\infty} F_i$ where $F_i$ is a closed set
in $X$ for each $i\in \mathbb{N}$ and for any compact set
$B\subseteq A$ there exists $i'\in \mathbb{N}$ such that
$B\subseteq F_{i'}$.
\end{definition}

\begin{definition} A space $X$ is called  {\it $k$-perfect} if every open subset is an
$k$-$F_{\sigma}$-subset of $X$.
\end{definition}

Clearly that every $k$-perfect space is perfect.

\begin{proposition}\label{pr1} Every perfectly normal space is $k$-perfect.
\end{proposition}

\begin{proof} Let $U$ be an open subset of a perfectly normal space
$X$. Then $U$ may be represented as
$U=\bigcup\limits_{i=1}^{\infty} F_i$ where $F_i$ is closed set of
$X$ for each $i\in \mathbb{N}$ and $F_i\subset F_{i+1}$. Since $X$
is a perfectly normal space, there exists a sequence $(U_i:i\in
\mathbb{N})$ of open sets of $X$ such that $F_i\subseteq
U_i\subseteq \overline{U_i}\subseteq U$ and
$\overline{U_i}\subseteq U_{i+1}$ for each $i\in \mathbb{N}$. It
follow that $U=\bigcup\limits_{i=1}^{\infty} \overline{U_i}$ and
for any compact set $B\subseteq U$ there exists $i'\in \mathbb{N}$
such that $B\subseteq \overline{U_{i'}}$.

\end{proof}

Note that, by Proposition \ref{pr1}, each cozero set is an
$k$-$F_{\sigma}$-set.

\bigskip

{\bf Question.} Is there a $k$-perfect which is not (perfectly)
normal space?

\bigskip

Denote by $\mathbb{S}$ the Sorgenfrey line.

\begin{proposition} The space $\mathbb{S}^2$ is perfect, but not
$k$-perfect.
\end{proposition}

\begin{proof} By Lemma 2.3 in \cite{hm}, $\mathbb{S}^2$ is
perfect.

Consider the open set $U=\{(x,y)\in \mathbb{S}^2 : -x<y \}\bigcup
\{(x,-x): x\in \mathbb{P}\}$ in the space $\mathbb{S}^2$.

Assume that $U$ is a $k$-$F_{\sigma}$-set of $\mathbb{S}^2$. Then
$U=\bigcup\limits_{i=1}^{\infty} F_i$ where $F_i$ is closed set in
$\mathbb{S}^2$ for each $i\in \mathbb{N}$ and for any compact set
$B\subseteq U$ there exists $i'\in \mathbb{N}$ such that
$B\subseteq F_{i'}$. For each point $p=(y,-y)\in \{(x,-x): x\in
\mathbb{P}\}$ we fix the compact set $Z_p=\{p\}\bigcup
\{(y+\frac{1}{n}, -y+\frac{1}{n}) : n\in \mathbb{N}\}$. Note that
$Z_p\subset U$ for each  $p\in \{(x,-x): x\in \mathbb{P}\}$. Since
$U=\bigcup\limits_{i=1}^{\infty} F_i$ there exists $k\in
\mathbb{N}$ such that $|\{p : Z_p\subset F_k\}|>\aleph_0$. Since
the set $\{p : Z_p\subset F_k\}$ is uncountable subset of
$\{(x,-x): x\in \mathbb{P}\}$ there is a accumulation point $z$ in
subspace $\{(x,-x): x\in \mathbb{R}\}$ of the space
$\mathbb{R}^2$. Clearly that for any neighborhood $O(z)$ of $z$ in
the space $\mathbb{S}^2$  we have that $O(z)\bigcap
\bigcup\limits_{Z_p\subset F_k} Z_p\neq \emptyset$. It follows
that $O(z)\bigcap F_k\neq \emptyset$. Hence $F_k$ is not closed
set in $\mathbb{S}^2$, a contradiction.

\end{proof}

\begin{theorem}\label{th13} Assume that $\Phi, \Psi\in \{\Gamma_{k}, \mathcal{K}\}$, $\star\in \{1,fin\}$, $X$ has the property $S_{\star}(\Phi,
\Psi)$ and $A$ is an $k$-$F_{\sigma}$-set. Then $A$ has the
property $S_{\star}(\Phi, \Psi)$.
\end{theorem}

\begin{proof} We prove the theorem for $\star=fin$, the other
proofs being similar.

Assume that $X$ has the property $S_{fin}(\Phi, \Psi)$ and $A$ is
an $k$-$F_{\sigma}$-set. Consider a sequence $(\mathcal{U}_i :
i\in \mathbb{N})$ of covers $A$ such that $\mathcal{U}_i\in
\Phi_A$ (where $\Phi_A$ is the  $\Phi$ family of covers of $A$)
for each $i\in \mathbb{N}$. Let $A=\bigcup\limits_{i=1}^{\infty}
F_i$ where $F_i$ is a closed set in $X$ for each $i\in \mathbb{N}$
and for any compact set $B\subseteq A$ there exists $i'\in
\mathbb{N}$ such that $B\subseteq F_{i'}$. Consider
$\mathcal{V}_i=\{(X\setminus F_i)\bigcup U : U\in \mathcal{U}_i
\}$ for each $i\in \mathbb{N}$.

We claim that $\mathcal{V}_i\in \Phi$ for each $i\in \mathbb{N}$.
Let $S$ be a compact subset of $X$. Then $S\bigcap F_i$ is a
compact subset of $A$ and hence there is $U\in \mathcal{U}_i$ such
that $S\bigcap F_i\subset U$. It follows that $S\subset
(X\setminus F_i)\bigcup U$ for $(X\setminus F_i)\bigcup U\in
\mathcal{V}_i$.

Since $X$ has the property $S_{fin}(\Phi, \Psi)$, there is a
sequence $(B_{i}: i\in \mathbb{N})$ of finite sets such that for
each $i$, $B_{i}\subset \mathcal{V}_i$, and
$\bigcup_{i\in\mathbb{N}}B_{i}\in \Psi$.

We claim that $\bigcup_{i\in\mathbb{N}}\{B_{i}\bigcap F_i :
B_{i}\bigcap F_i\subset \mathcal{U}_i$, $i\in \mathbb{N}\}\in
\Psi_A$. Let $B$ be a compact subset of $A$ then there exists
$i'\in \mathbb{N}$ such that $B\subseteq F_{i'}$. Since
$\bigcup_{i\in\mathbb{N}}B_{i}$ is a large cover of $X$ there is
$k\in \mathbb{N}$  and $V_k\in B_{k}\subset \mathcal{V}_k$ such
that $k>i'$ and $B\subset V_k$. But $V_k=(X\setminus F_k)\bigcup
U_k$ for $U_k\in \mathcal{U}_k$. Since $k>i'$, $B\subset U_k$. It
follows that $A$ has the property $S_{fin}(\Phi, \Psi)$.

\end{proof}

\begin{proposition} Assume that $\Phi, \Psi\in \{\Gamma_k, \mathcal{K}\}$, $\star\in \{1,fin\}$. Then for a $k$-perfect 
space $X$ the following statements are equivalent:

\begin{enumerate}

\item $X$ satisfies $S_{\star}(\Phi, \Psi)$;

\item Each open set $Y\subset X$ has the property $S_{\star}(\Phi,
\Psi)$.

\end{enumerate}

\end{proposition}

By Proposition  \ref{pr1}, we have the following result.

\begin{proposition} Assume that $\Phi, \Psi\in \{\Gamma_k, \mathcal{K}\}$, $\star\in \{1,fin\}$. Then for a perfectly normal
space $X$ the following statements are equivalent:

\begin{enumerate}

\item $X$ satisfies $S_{\star}(\Phi, \Psi)$;

\item Each open set $Y\subset X$ has the property $S_{\star}(\Phi,
\Psi)$.

\end{enumerate}

\end{proposition}

\begin{theorem} Assume that $\Phi, \Psi\in \{\Gamma_k, \mathcal{K}\}$, $\star\in \{1,fin\}$.  Then for a
$k$-perfect space $X$ the following statements are equivalent:

\begin{enumerate}

\item $X$ satisfies $S_{\star}(\Phi, \Psi)$;

\item $(2^X,{\bf F^+})$ satisfies $S_{\star}(\mathcal{D}_{\Phi},
\mathcal{D}_{\Psi})$;

\item For each $E\in 2^X$, $(2^X,{\bf F^+})$ satisfies
$S_{\star}(\Phi_{E}, \Psi_{E})$.

\end{enumerate}

\end{theorem}

By Proposition  \ref{pr1}, we have the following result.

\begin{corollary} Assume that $\Phi, \Psi\in \{\Gamma_k, \mathcal{K}\}$, $\star\in \{1,fin\}$.  Then for a perfectly normal
space $X$ the following statements are equivalent:

\begin{enumerate}

\item $X$ satisfies $S_{\star}(\Phi, \Psi)$;

\item $(2^X,{\bf F^+})$ satisfies $S_{\star}(\mathcal{D}_{\Phi},
\mathcal{D}_{\Psi})$;

\item  For each $E\in 2^X$, $(2^X,{\bf F^+})$ satisfies
$S_{\star}(\Phi_{E}, \Psi_{E})$.

\end{enumerate}

\end{corollary}

\section{Bitopological space $(2^X,{\bf F^+}, {\bf Z^+})$}

\begin{theorem} Assume that $\Phi\in \{\Gamma_k, \mathcal{K}\}$, $\Psi\in \{\Gamma, \Omega\}$, $\star\in \{1,fin\}$.  Then for a
space $X$ the following statements are equivalent:

\begin{enumerate}

\item $X$ satisfies $S_{\star}(\Phi, \Psi)$;

\item $(2^X,{\bf F^+}, {\bf Z^+})$ satisfies
$S_{\star}(\mathcal{D}^{\bf F^+}_{\Phi}, \mathcal{D}^{\bf
Z^+}_{\Psi})$.

\end{enumerate}

\end{theorem}

\begin{proof} The proof is similar to the proof of Theorem
\ref{th11}.

\end{proof}

\begin{corollary}(Theorem 6 in \cite{mkm1}) For a space $X$ the following are equivalent:

\begin{enumerate}

\item $X$ satisfies $S_{1}(\mathcal{K}, \Omega)$;

\item $(2^X,{\bf F^+}, {\bf Z^+})$ satisfies
$S_{1}(\mathcal{D}^{\bf F^+}_{\Omega}, \mathcal{D}^{\bf
Z^+}_{\Omega})$.

\end{enumerate}

\end{corollary}

\begin{corollary}(Theorem 14 in \cite{mkm1}) For a space $X$ the following are equivalent:

\begin{enumerate}

\item $X$ satisfies $S_{fin}(\mathcal{K}, \Omega)$;

\item $(2^X,{\bf F^+}, {\bf Z^+})$ satisfies
$S_{fin}(\mathcal{D}^{\bf F^+}_{\Omega}, \mathcal{D}^{\bf
Z^+}_{\Omega})$.

\end{enumerate}

\end{corollary}

\begin{theorem} Assume that $\Phi\in \{\Gamma_k, \mathcal{K}\}$, $\Psi\in \{\Gamma, \Omega\}$, $\star\in \{1,fin\}$.  Then for a space $X$ the following statements are equivalent:

\begin{enumerate}

\item Each open set $Y\subset X$ has the property $S_{\star}(\Phi,
\Psi)$;

\item  For each $E\in 2^X$, $(2^X,{\bf F^+}, {\bf Z^+})$ satisfies
$S_{\star}(\Phi^{\bf F^+}_{E}, \Psi^{\bf Z^+}_{E})$.

\end{enumerate}

\end{theorem}

\begin{proof}  The proof is similar to the proof of Theorem
\ref{th12}.

\end{proof}

\begin{corollary}(Theorem 31 in \cite{cmkm}) For a space $X$ the
following are equivalent:

\begin{enumerate}

\item Each open set $Y\subset X$ has the property $S_{1}(\Gamma_k,
\Omega)$;

\item  For each $E\in 2^X$, $(2^X,{\bf F^+}, {\bf Z^+})$ satisfies
$S_{1}(\Gamma^{\bf F^+}_{E}, \Omega^{\bf Z^+}_{E})$.

\end{enumerate}

\end{corollary}

\begin{corollary}(Theorem 33 in \cite{cmkm}) For a space $X$ the
following are equivalent:

\begin{enumerate}

\item Each open set $Y\subset X$ has the property $S_{1}(\Gamma_k,
\Gamma)$;

\item  For each $E\in 2^X$, $(2^X,{\bf F^+}, {\bf Z^+})$ satisfies
$S_{1}(\Gamma^{\bf F^+}_{E}, \Gamma^{\bf Z^+}_{E})$.

\end{enumerate}

\end{corollary}

\begin{theorem} Assume that $\Phi\in \{\Gamma_{k}, \mathcal{K}\}$, $\Psi\in \{\Gamma, \Omega\}$, $\star\in \{1,fin\}$, $X$ has the property $S_{\star}(\Phi,
\Psi)$ and $A$ is an $k$-$F_{\sigma}$-set. Then $A$ has the
property $S_{\star}(\Phi, \Psi)$.
\end{theorem}

\begin{proof}  The proof is similar to the proof of Theorem
\ref{th13}.
\end{proof}

\begin{proposition} Assume that $\Phi\in \{\Gamma_{k}, \mathcal{K}\}$, $\Psi\in \{\Gamma, \Omega\}$, $\star\in \{1,fin\}$.
Then for a $k$-perfect space $X$ the following statements are
equivalent:

\begin{enumerate}

\item $X$ satisfies $S_{\star}(\Phi, \Psi)$;

\item Each open set $Y\subset X$ has the property $S_{\star}(\Phi,
\Psi)$.

\end{enumerate}

\end{proposition}

\begin{theorem} Assume that $\Phi\in \{\Gamma_{k}, \mathcal{K}\}$, $\Psi\in \{\Gamma, \Omega\}$, $\star\in \{1,fin\}$.  Then for a
$k$-perfect space $X$ the following statements are equivalent:

\begin{enumerate}

\item $X$ satisfies $S_{\star}(\Phi, \Psi)$;

\item $(2^X,{\bf F^+}, {\bf Z^+})$ satisfies
$S_{\star}(\mathcal{D}_{\Phi}, \mathcal{D}_{\Psi})$;

\item For each $E\in 2^X$, $(2^X,{\bf F^+}, {\bf Z^+})$ satisfies
$S_{\star}(\Phi_{E}, \Psi_{E})$.

\end{enumerate}

\end{theorem}

\bibliographystyle{model1a-num-names}
\bibliography{<your-bib-database>}

\begin{thebibliography}{10}




\bibitem{cmkm}
A. Caserta, G. Di Maio, Lj.D.R. Ko$\check{c}$inac and E.
Meccariello, \textit{Applications of $k$-covers II}, Topology and
its Applications, n.~153, (2006), 3277--3293.

\bibitem{fe}
J. Fell, \textit{A Hausdorff topology for the closed subsets of a
locally compact non-Hausdorff spaces}, Proceedings of the American
Mathematical Society, 13, (1962), 472--476.



\bibitem{jmss}
W. Just, A.W. Miller, M. Scheepers, P.J. Szeptycki, \textit{The
combinatorics of open covers, II}, Topology and its Applications,
73, (1996), 241--266.


\bibitem{koc2}
Lj.D.R. Ko$\check{c}$inac, \textit{Selected results on selection
principles}, in: Proceedings of the 3rd Seminar on Geometry and
Topology (Sh. Rezapour, ed.), July 15--17, Tabriz, Iran, (2004),
71--104.


\bibitem{koc1}
Lj.D.R. Ko$\check{c}$inac, \textit{$\gamma$-sets, $\gamma_k$-sets
and hyperspaces}, Mathematica Balkanica 19 (2005), 109-118.


\bibitem{zl}
Z. Li, \textit{Selection principles of the Fell topology and the
Vietoris topology}, Topology and its Applications, 212, (2016),
90--104.




\bibitem{mkn}
G.Di Maio, Lj.D.R. Ko$\check{c}$inac, T.Nogura,
\textit{Convergence properties of hyperspaces}, J. Korean Math.
Soc. 44 (2007), n.4, 845-854.

\bibitem{mkm1}
G.Di Maio, Lj.D.R. Ko$\check{c}$inac, E.Meccariello,
\textit{Selection principles and hyperspace topologies}, Topology
and its Applications, 153, (2005), 912--923.


\bibitem{mm1}
M. Mr$\check{s}$evi$\acute{c}$, M. Jeli$\acute{c}$,
\textit{Selection principles in hyperspaces with generalized
Vietoris topologies}, Topology and its Applications, 156:1,
(2008), 124--129.

\bibitem{hm}
R. W. Heath, E.A. Michael, \textit{A property of the Sorgenfrey
line}, Compositio Mathematica, 23:2, (1971), 185--188.


\bibitem{ot}
T. Orenshtein, B. Tsaban, \textit{Linear $\sigma$-additivity and
some applications}, Transactions of the AMS, 363:7, (2011),
3621--3637.


\bibitem{os1}
A.V. Osipov, \textit{Application of selection principles in the
study of the properties of function spaces}, Acta Math. Hungar.,
154(2), (2018), 362-377.

\bibitem{os2}
A.V. Osipov, \textit{Classification of selectors for sequences of
dense sets of $C_p(X)$}, Topology and its Applications, 242,
(2018), 20-32.


\bibitem{os4}
A.V. Osipov, \textit{The functional characteristics of the
Rothberger and Menger properties}, Topology and its Applications,
243, (2018), 146--152.

\bibitem{os5}
A.V. Osipov, S. \"{O}z\c{c}a\u{g}, \textit{Variations of selective
separability and tightness in function spaces with set-open
topologies}, Topology and its Applications, 217, (2017), p.38--50.


\bibitem{sak}
M. Sakai, \textit{Selective separability of Pixley-Roy
hyperspaces}, Topology and its Applications, 159, (2012),
1591--1598.



\bibitem{sash}
M. Sakai, M. Scheepers, \textit{The combinatorics of open covers},
Recent Progress in General Topology III, (2013), Chapter, p.
751--799.






\bibitem{sch3}
M. Scheepers, \textit{Combinatorics of open covers (I): Ramsey
Theory}, Topology and its Applications, 69, (1996), p. 31--62.



\bibitem{ts1}
B. Tsaban, \textit{Some New Directions in Infinite-combinatorial
Topology}, in: Bagaria J., Todorcevic S. (eds) Set Theory. Trends
in Mathematics. Birkh$\ddot{a}$user Basel.(2006).


\end{thebibliography}







\end{document}